\documentclass[portuges,12pt,letter]{article}
\usepackage[centertags]{amsmath}
\usepackage{amsfonts}
\usepackage{newlfont}
\usepackage{amscd}
\usepackage{graphics}
\usepackage{epsfig}
\usepackage{indentfirst}
\usepackage{amsxtra}
\usepackage[latin1]{inputenc}
\usepackage{amssymb, amsmath}
\usepackage{amsthm}
\usepackage[mathscr]{eucal}

\newtheorem{thm}{Theorem}[section]

\newtheorem{pro}[thm]{Proposition}
\newtheorem{dfn}[thm]{Definition}

\author{Fabio Silva Botelho \\ Department of Mathematics \\  Federal University of Santa Catarina, UFSC \\
Florian\'{o}polis, SC - Brazil}

\title{\bf  On Lagrange multiplier theorems for non-smooth optimization for a large class of variational models in Banach spaces}
\date{}
\begin{document}
\maketitle

\abstract{This article develops optimality conditions for a large class of non-smooth
variational models. The main results are based on standard tools of functional analysis and
calculus of variations. Firstly we address a model with equality constraints and, in a second
step, a more general model with equality and inequality constraints, always in a general Banach
space context. We highlight  the results in general are well known, however, some novelties are introduced related to the proof procedures,
which are in general softer than those concerning the present literature.}
\\
\\
{\bf Key words}: Non-smooth optimization, Lagrange multiplier theorems, Equality and inequality constraints.
\\
\\
{\bf MSC Class}: 49K27

\section{Introduction}
In this article we present  Lagrange multiplier results for non-smooth variational optimization, firstly for an equality constraints model and, in a subsequent step, for a more general problem involving equality and inequality
constraints. We emphasize the results are in general well known, however some novelties are introduced concerning the related proofs. It is also worth mentioning the results are rather general and are suitable  in a Banach space context.

Moreover, the main references for this article are \cite{120,700}. Other important reference is \cite{50}.

We also highlight  specific  details on the function spaces addressed and concerning functional analysis and Lagrange multiplier basic results may be found in \cite{1,510,27,8,50,520,120,700}.

Related subjects are addressed in \cite{51,52}. Specifically in \cite{52}, the authors propose an augmented Lagrangian method for the solution of constrained optimization problems suitable for a large class of variational models.

At this point, we highlight the main novelties mentioned in the abstract are specified in the first three paragraphs of section \ref{ch12} and are
applied in the statements and proofs of Theorems \ref{uk12} and \ref{us12}.

Finally, fundamental results on the calculus of variations are addressed in \cite{1210}.

We start with some preliminary results and  basic definitions.
The first result we present is the Hahn-Banach Theorem in its analytic form.
Concerning our context, we have assumed the hypothesis the space $U$ is a Banach space but indeed such a result is much more general.

\begin{thm}[The Hahn-Banach theorem]\label{2.2.1}
Let $U$ be a Banach space. Consider a functional $p: U \rightarrow \mathbb{R}$ such that
\begin{gather}
p(\lambda u) =\lambda p(u), \forall u \in U, \lambda > 0,
\end{gather}
and
\begin{gather}
p(u+v) \leq p(u)+p(v), \forall u,v \in U.
\end{gather}
Let $V \subset U$ be a proper subspace of  $U$ and let $g:V \rightarrow
\mathbb{R}$ be a linear functional such that
\begin{gather}
g(u) \leq p(u), \forall u \in V.
\end{gather}
Under such hypotheses, there exists a linear functional $f:U \rightarrow \mathbb{R}$
such that
\begin{gather}
g(u)=f(u), \forall u \in V,
\end{gather}
and
\begin{gather}
f(u) \leq  p(u), \forall u \in U.
\end{gather}
\end{thm}
For a proof, please see  \cite{27,120,700}.

Here we introduce the definition of topological dual space.
\begin{dfn}[Topological dual spaces] Let $U$ be a Banach space. We shall define its dual topological space, as the set of all linear continuous functionals defined on $U$.  We  suppose such a dual space of $U$, may be represented by another Banach space $U^*$, through a bilinear form $ \langle \cdot , \cdot \rangle_U: U \times U^*
\rightarrow \mathbb{R}$ (here we are referring  to standard representations of dual spaces of Sobolev and Lebesgue spaces). Thus, given $f:U \rightarrow \mathbb{R}$
linear and continuous, we assume the existence of a unique  $u^* \in U^*$ such that
\begin{gather}
f(u)=\langle u, u^* \rangle_U, \forall u \in U.
\end{gather}
The norm of $f$ , denoted by $\|f\|_{U^*}$, is defined as
\begin{gather}
\|f\|_{U^*}=\sup_{u \in U}\{|\langle u, u^* \rangle_U| \; : \; \|u\|_U \leq 1\}\equiv \|u^*\|_{U^*}.
\end{gather}
\end{dfn}

At this point we present the Hahn-Banach Theorem in its geometric form.

\begin{thm}[The Hahn-Banach theorem, the geometric form]\label{2.2.2.2}
Let $U$ be a Banach space and let $A,B \subset U$ be two non-empty, convex sets such that $A \cap B= \emptyset $ and
$A$ is open. Under such hypotheses, there exists a closed hyperplane which separates $A$ and $B$, that is, there exist $\alpha \in \mathbb{R}$ and $u^* \in U^*$ such that $u^* \neq \mathbf{0}$
and
$$\langle u,u^*\rangle_U \leq \alpha \leq \langle v, u^* \rangle_U,\; \forall u \in A,\; v \in B.$$
\end{thm}
For a proof, please see  \cite{27,120,700}

Another important definition, is the one concerning locally Lipschitz functionals.

\begin{dfn} Let $U$ be a Banach space and let $F:U \rightarrow \mathbb{R}$ be a functional. We say that $F$
is locally Lipschitz   at $u_0 \in U$ if there exist $r>0$ and $K>0$ such that
$$|F(u)-F(v)| \leq K \|u-v\|_U,\; \forall u,v \in B_r(u_0).$$
\end{dfn}
In this definition, we have denoted
$$B_r(u_0)=\{v \in U \;:\; \|u_0-v\|_U < r\}.$$

The next definition is established similarly as those found in the reference \cite{50}. More specifically, it is similar as the definition of generalized directional derivative found in section 10.1 at page 194, in reference \cite{50}.

\begin{dfn} Let $U$ be a Banach space and let $F:U \rightarrow \mathbb{R}$ be a locally Lipschitz  functional
at $u \in U$. Let $\varphi \in U$. Under such statements, we define

$$H_u(\varphi)=\sup_{(\{u_n\},\{t_n\}) \subset U \times \mathbb{R}^+} \left\{\limsup_{n \rightarrow \infty}
\frac{F(u_n+t_n \varphi)-F(u_n)}{t_n}\;:\; u_n \rightarrow u \text{ in } U,\; t_n \rightarrow 0^+\right\}.$$

We also define the generalized local sub-gradient set of $F$ at $u$, denoted by $\partial^0 F(u)$, by
$$\partial^0F(u)=\{ u^* \in U^*\;:\; \langle \varphi,u^*\rangle_U \leq H_u(\varphi),\; \forall \varphi \in U\}.$$
\end{dfn}

We also highlight such a last definition of generalized local sub-gradient is similar to the definition of generalized gradient, which may be found in section 10.13, at page 196, in the book \cite{50}.

In the next lines we prove some relevant auxiliary results.

\begin{pro} Considering the context of the last two definitions, we have
\begin{enumerate}
\item $$H_u(\varphi_1+\varphi_2) \leq H_u(\varphi_1)+H_u(\varphi_2),\; \forall \varphi_1, \varphi_2 \in U.$$

\item $$H_u(\lambda \varphi)=\lambda H_u(\varphi),\; \forall \lambda>0,\; \varphi \in U.$$
\end{enumerate}
\end{pro}
\begin{proof}

Let $\varphi_1,\varphi_2 \in U.$

Observe that
\begin{eqnarray}
&&H_u(\varphi_1+\varphi_2) \nonumber \\ &=&
\sup_{(\{u_n\},\{t_n\}) \subset U \times \mathbb{R}^+} \left\{\limsup_{n \rightarrow \infty}
\frac{F(u_n+t_n (\varphi_1+\varphi_2))-F(u_n)}{t_n}\;:\; u_n \rightarrow u \text{ in } U,\; t_n \rightarrow 0^+\right\}
\nonumber \\ &=&
\sup_{(\{u_n\},\{t_n\}) \subset U \times \mathbb{R}^+} \left\{\limsup_{n \rightarrow \infty}
\frac{F(u_n+t_n (\varphi_1+\varphi_2)-F(u_u+t_n \varphi_2)+F(u_n+t_n \varphi_2))-F(u_n)}{t_n}\right.
\nonumber \\ &&\;:\;\left. u_n \rightarrow u \text{ in } U,\; t_n \rightarrow 0^+\right\}
\nonumber \\ &\leq& \sup_{(\{v_n\},\{t_n\}) \subset U \times \mathbb{R}^+} \left\{\limsup_{n \rightarrow \infty}
\frac{F(v_n+t_n \varphi_1)-F(v_n)}{t_n}\;:\; v_n \rightarrow u \text{ in } U,\; t_n \rightarrow 0^+\right\}
 \nonumber \\ &&+\sup_{(\{u_n\},\{t_n\}) \subset U \times \mathbb{R}^+} \left\{\limsup_{n \rightarrow \infty}\frac{F(u_n+t_n \varphi_2)-F(u_n)}{t_n}\;:\; u_n \rightarrow u \text{ in } U,\; t_n \rightarrow 0^+\right\}
 \nonumber \\ &=& H_u(\varphi_1)+H_u(\varphi_2).
 \end{eqnarray}

Let $\varphi \in U$ and $\lambda>0$.

Thus,
\begin{eqnarray}
&&H_u(\lambda \varphi) \nonumber \\ &=&
\sup_{(\{u_n\},\{t_n\}) \subset U \times \mathbb{R}^+} \left\{\limsup_{n \rightarrow \infty}
\frac{F(u_n+t_n (\lambda \varphi))-F(u_n)}{t_n}\;:\; u_n \rightarrow u \text{ in } U,\; t_n \rightarrow 0^+\right\}
\nonumber \\ &=&\lambda \sup_{(\{u_n\},\{t_n\}) \subset U \times \mathbb{R}^+} \left\{\limsup_{n \rightarrow \infty}
\frac{F(u_n+t_n (\lambda \varphi))-F(u_n)}{\lambda t_n}\;:\; u_n \rightarrow u \text{ in } U,\; t_n \rightarrow 0^+\right\}\nonumber \\ &=&\lambda \sup_{(\{u_n\},\{\hat{t}_n\}) \subset U \times \mathbb{R}^+} \left\{\limsup_{n \rightarrow \infty}
\frac{F(u_n+\hat{t}_n (\varphi))-F(u_n)}{\hat{t}_n}\;:\; u_n \rightarrow u \text{ in } U,\; \hat{t}_n \rightarrow 0^+\right\} \nonumber \\ &=& \lambda H_u(\varphi).
\end{eqnarray}

The proof is complete.
\end{proof}
\section{ The Lagrange multiplier theorem for equality constraints and non-smooth optimization}\label{ch12}
In this section we state and prove a Lagrange multiplier theorem for non-smooth optimization.
This first one is related to equality constraints.

Here we refer to a related result in the Theorem 10.45 at page 220, in the book \cite{50}. We emphasize that in such a result, in this mentioned book, the author assumes the function which defines the constraints to be continuously differentiable in a neighborhood of the  point in question.

Anyway, in our next result, we do not assume such hypothesis. Indeed, our hypotheses are different and in some sense weaker. More specifically, we assume the continuity of the Frech\'{e}t derivative $G'(u)$ of a concerning constraint $G(u)$ only at the optimal point $u_0$ and not necessarily in a neighborhood, as properly indicated in the next lines.

\begin{thm}\label{uk12} Let $U$ and $Z$ be Banach spaces. Assume $u_0$ is a local minimum of $F(u)$ subject to $G(u)=\theta,$
where $F: U \rightarrow \mathbb{R}$ is locally Lipschitz  at $u_0$ and $G:U \rightarrow Z$ is a  Fr\'{e}chet differentiable transformation such that  $G'(u_0)$  maps $U$ onto $Z$.
Finally, assume  there exist  $\alpha>0$ and $K>0$ such that if $\|\varphi\|_U<\alpha$ then, $$ \|G'(u_0+\varphi)-G'(u_0)\| \leq K \|\varphi\|_U.$$
Under such assumptions, there exists $z^*_0 \in Z^*$ such that $$\theta \in \partial^0F(u_0)+(G'(u_0)^*)(z_0^*),$$ that is, there exist $u^* \in \partial^0F(u_0)$ and $z_0^* \in Z^*$ such that
$$u^*+[G'(u_0)]^*(z_0^*)= \theta,$$
so that,
$$ \langle \varphi, u^* \rangle_U+ \langle G'(u_0) \varphi, z_0^* \rangle_Z=0, \forall \varphi \in U.$$
\end{thm}
\begin{proof}
Let $\varphi \in U$ be such that $$G'(u_0) \varphi= \theta.$$

From the proof of Theorem 11.3.2 at page 292, in \cite{120}, there exist $\varepsilon_0>0$, $K_1>0$ and
$$\{\psi_0(t),\; 0<|t|< \varepsilon_0\}\subset U$$ such that

$$\|\psi_0(t)\|_U \leq K_1,\; \forall 0<|t|< \varepsilon_0,$$
and
$$G(u_0+t \varphi+t^2 \psi_0(t))=\theta,\; \forall 0<|t|<\varepsilon_0.$$

From this and the hypotheses on $u_0$, there exists $0< \varepsilon_1< \varepsilon_0$ such that

$$F(u_0+t \varphi+t^2 \psi_0(t)) \geq F(u_0),\; \forall 0<|t|< \varepsilon_1,$$

so that
$$\frac{F(u_0+t \varphi+t^2 \psi_0(t))-F(u_0)}{t} \geq 0,\; \forall 0<t< \varepsilon_1.$$

Hence,

\begin{eqnarray}
&&0 \leq \frac{F(u_0+t \varphi+t^2 \psi_0(t))-F(u_0)}{t} \nonumber \\ &=&
\frac{F(u_0+t \varphi+t^2 \psi_0(t))-F(u_0+t^2\psi_0(t))+F(u_0+t^2\psi_0(t))-F(u_0)}{t}
\nonumber \\ &\leq& \frac{F(u_0+t \varphi+t^2 \psi_0(t))-F(u_0+t^2\psi_0(t))}{t}+K t\|\psi_0(t)\|_U,\;\forall 0<t< \min\{r,\varepsilon_1\}.
\end{eqnarray}

From this, we obtain
\begin{eqnarray}
&&0 \leq \limsup_{t \rightarrow 0^+}\frac{F(u_0+t \varphi+t^2 \psi_0(t))-F(u_0)}{t}
\nonumber \\ &=& \limsup_{t \rightarrow 0^+}\frac{F(u_0+t \varphi+t^2 \psi_0(t))-F(u_0+t^2\psi_0(t))+F(u_0+t^2\psi_0(t))-F(u_0)}{t}\nonumber \\
&\leq& \limsup_{t \rightarrow 0^+}\frac{F(u_0+t \varphi+t^2 \psi_0(t))-F(u_0+t^2\psi_0(t))}{t}+\limsup_{t \rightarrow 0^+}K t\|\psi_0(t)\|_U \nonumber \\ &=&
\limsup_{t \rightarrow 0^+}\frac{F(u_0+t \varphi+t^2 \psi_0(t))-F(u_0+t^2\psi_0(t))}{t}
\nonumber \\ &\leq& H_{u_0}(\varphi).
\end{eqnarray}

Summarizing,
$$H_{u_0}(\varphi) \geq 0,\; \forall \varphi \in N(G'(u_0)).$$

Hence,
$$H_{u_0}(\varphi) \geq 0=\langle \varphi, \theta \rangle_U,\; \forall \varphi \in N(G'(u_0)).$$

From the Hahn-Banach Theorem, the functional $$f\equiv 0$$ defined on $N(G'(u_0))$ may be extended to $U$ through a linear functional $f_1:U \rightarrow \mathbb{R}$ such that
$$f_1(\varphi)=0,\; \forall \varphi \in N[G'(u_0)]$$ and
$$f_1(\varphi) \leq H_{u_0}(\varphi),\; \forall \varphi \in U.$$

Since from the local Lipschitz property $H_{u_0}$ is bounded, so is $f_1$.

Therefore, there exists $u^* \in U^*$ such that
$$ f_1(\varphi)=\langle \varphi,u^* \rangle_U \leq H_{u_0}(\varphi), \forall \varphi \in U,$$
so that $$u^* \in \partial^0F(u_0).$$

Finally, observe that $$\langle \varphi,u^* \rangle_U=0,\; \forall \varphi \in N(G'(u_0)).$$

Since $G'(u_0)$ is onto (closed range), from a well known result for linear operators, we have that
$$u^* \in R[G'(u_0)^*].$$

Thus, there exists, $z_0^* \in Z^*$ such that
 $$u^*=[G'(u_0)^*](-z_0^*),$$ so that
 $$u^*+[G'(u_0)^*](z_0^*)=\theta.$$

 From this, we obtain
 $$\langle \varphi,u^* \rangle_U+\langle \varphi,[ G'(u_0)^*](z_0^*) \rangle_U=0,$$
 that is,
 $$\langle \varphi,u^* \rangle_U+\langle G'(u_0) \varphi,(z_0^*) \rangle_Z=0,\; \forall \varphi \in U$$

 The proof is complete.

 \end{proof}
 \section{The Lagrange multiplier theorem for equality and inequality constraints for non-smooth optimization}
In this section we develop a rigorous result concerning the Lagrange multiplier theorem for the case involving equalities and
inequalities.
\begin{thm}\label{us12}Let $U,Z_1,Z_2$ be Banach spaces. Consider a cone $C$ in $Z_2$ (as  specified at Theorem 11.1 in \cite{120})  such that if $z_1 \leq \theta$
and $z_2 < \theta$ then $z_1+z_2 < \theta$, where $z \leq \theta$ means that $z \in -C$ and $z<\theta$ means that
$z \in (-C)^\circ$. The concerned order is supposed to be also that if $z< \theta$, $z^* \geq \theta^*$ and $z^* \neq \theta$ then
 $\langle z,z^* \rangle_{Z_2}<0.$ Furthermore, assume $u_0 \in U$ is a point of
local minimum for  $F:U \rightarrow \mathbb{R}$ subject to $G_1(u)=\theta$ and $G_2(u) \leq \theta$, where $G_1:U \rightarrow Z_1$,
$G_2:U \rightarrow Z_2$ are Fr\'{e}chet differentiable transformations and $F$ locally Lipschitz  at $u_0 \in U$. Suppose also $G_1'(u_0)$ is onto  and that there exist
$\alpha>0, K>0$ such that if $\|\varphi\|_U< \alpha$ then
$$\|G'_1(u_0+\varphi)-G'_1(u_0)\| \leq K\|\varphi\|_U.$$

Finally, suppose there exists $\varphi_0 \in U$ such that $$G_1'(u_0)\cdot \varphi_0=\theta$$ and
$$G_2'(u_0) \cdot \varphi_0 < \theta.$$

Under such hypotheses, there exists  a Lagrange multiplier $z_0^*=(z_1^*,z_2^*) \in Z_1^* \times Z_2^*$ such that
$$\theta \in \partial^0F(u_0)+[G_1'(u_0)^*](z_1^*)+[G_2'(u_0)^*](z_2^*),$$
$$z_2^* \geq \theta^*,$$ and
$$\langle G_2(u_0),z_2^* \rangle_{Z_2}=0,$$
that is, there exists $u^* \in \partial^0F(u_0)$  and a Lagrange multiplier $z_0^*=(z_1^*,z_2^*) \in Z_1^* \times Z_2^*$ such that
$$u^*+[G_1'(u_0)]^*(z_1^*)+[G_2'(u_0)]^*(z_2^*)=\theta,$$
so that $$\langle \varphi,u^*\rangle_U+\langle \varphi, G_1'(u_0)^*(z_1^*)\rangle_U+\langle \varphi, G_2'(u_0)^*(z_2^*)\rangle_U=0,$$
that is,
$$\langle \varphi,u^*\rangle_U+\langle G_1'(u_0)\varphi, z_1^*\rangle_{Z_1}+\langle G_2'(u_0)\varphi, z_2^*\rangle_{Z_2}=0,\; \forall \varphi \in U.$$
\end{thm}
\begin{proof}
Let $\varphi \in U$ be such that
$$G_1'(u_0) \cdot \varphi=\theta$$ and
$$G_2'(u_0) \cdot \varphi=v-\lambda G_2(u_0),$$
for some $v \leq \theta$ and $\lambda \geq 0$.

For $\alpha \in (0,1)$ define
$$\varphi_\alpha=\alpha \varphi_0+(1-\alpha)\varphi.$$

Observe that $G_1(u_0)=\theta$ and $G_1'(u_0) \cdot \varphi_\alpha=\theta$ so that as in the proof of the Lagrange multiplier Theorem 11.3.2 in \cite{120},
we may find $K_1>0$, $\varepsilon>0$ and $\psi_0^\alpha(t)$ such that
$$G_1(u_0+t\varphi_\alpha+t^2\psi_0^\alpha(t))=\theta,\; \forall |t|< \varepsilon, \forall \alpha \in (0,1)$$
and
$$\|\psi_0^\alpha(t)\|_U < K_1, \forall |t|< \varepsilon, \; \forall \alpha \in (0,1).$$
Observe that \begin{eqnarray}\label{us569}
&&G_2'(u_0)\cdot \varphi_\alpha \nonumber\\ &=& \alpha G_2'(u_0) \cdot \varphi_0+(1-\alpha)G_2'(u_0)\cdot \varphi \nonumber
\\ &=& \alpha G_2'(u_0) \cdot \varphi_0+(1-\alpha)(v-\lambda G_2(u_0))\nonumber \\ &=&
\alpha G_2'(u_0) \cdot \varphi_0+(1-\alpha)v-(1-\alpha)\lambda G_2(u_0)) \nonumber \\ &=& v_0-\lambda_0 G_2(u_0), \end{eqnarray}
where,
$$\lambda_0=(1-\alpha)\lambda,$$ and
$$v_0=\alpha G_2'(u_0) \cdot \varphi_0+(1-\alpha)v< \theta.$$

Hence, for $t>0$
$$G_2(u_0+t \varphi_\alpha+t^2\psi_0^\alpha(t))=G_2(u_0)+G_2'(u_0)\cdot (t\varphi_\alpha +t^2\psi_0^\alpha(t))+r(t),$$
where $$\lim_{t \rightarrow 0^+} \frac{\|r(t)\|}{t}=0.$$

Therefore from (\ref{us569}) we obtain
$$G_2(u_0+t \varphi_\alpha+t^2\psi_0^\alpha(t))=G_2(u_0)+tv_0-t\lambda_0G_2(u_0)+r_1(t),$$
where $$\lim_{t \rightarrow 0^+} \frac{\|r_1(t)\|}{t}=0.$$
Observe that there exists $\varepsilon_1>0$ such that if $0<t < \varepsilon_1< \varepsilon,$ then
$$v_0+\frac{r_1(t)}{t}< \theta,$$
and
$$G_2(u_0)-t\lambda_0G_2(u_0)=(1-t \lambda_0)G_2(u_0) \leq \theta.$$

Hence $$G_2(u_0+t \varphi_\alpha+t^2\psi_0^\alpha(t)) < \theta, \; \text{ if } 0<t < \varepsilon_1.$$

From this there exists $0<\varepsilon_2< \varepsilon_1$ such that
$$
F(u_0+t \varphi_\alpha+t^2\psi_0^\alpha(t))\geq F(u_0), \forall 0<t< \varepsilon_2,\; \alpha \in (0,1).$$

In particular
$$F(u_0+t \varphi_t+t^2\psi_0^t(t))\geq F(u_0), \forall 0<t< \min\{1,\varepsilon_2\},$$
so that
$$\frac{F(u_0+t \varphi_t+t^2\psi_0^t(t))- F(u_0)}{t} \geq 0,\; \forall 0<t< \min\{1,\varepsilon_2\},$$
that is,
$$\frac{F(u_0+t \varphi+t^2(\psi_0^t(t)+\varphi_0-\varphi))- F(u_0)}{t} \geq 0,\; \forall 0<t< \min\{1,\varepsilon_2\}.$$

From this we obtain,
\begin{eqnarray}
&&0 \leq \limsup_{t \rightarrow 0^+}\frac{F(u_0+t \varphi+t^2(\psi_0^t(t)+\varphi_0-\varphi))-F(u_0)}{t}
\nonumber \\ &=& \limsup_{t \rightarrow 0^+}\left(\frac{F(u_0+t \varphi+t^2 (\psi_0^t(t)+\varphi_0-\varphi))-F(u_0+t^2(\psi_0^t(t)+\varphi_0-\varphi))}{t}\right.
\nonumber \\ && \left.+\frac{F(u_0+t^2(\psi_0^t(t)+\varphi_0-\varphi))-F(u_0)}{t}\right)\nonumber \\
&\leq& \limsup_{t \rightarrow 0^+}\frac{F(u_0+t \varphi+t^2 (\psi_0^t(t)+\varphi_0-\varphi))-F(u_0+t^2(\psi_0^t(t)+\varphi_0-\varphi))}{t}
\nonumber \\ &&+\limsup_{t \rightarrow 0^+}K t\|\psi_0^t(t)+\varphi_0-\varphi\|_U \nonumber \\ &=&
\limsup_{t \rightarrow 0^+}\frac{F(u_0+t \varphi+t^2 (\psi_0^t(t)+\varphi_0-\varphi))-F(u_0+t^2(\psi_0^t(t)+\varphi_0-\varphi))}{t}
\nonumber \\ &\leq& H_{u_0}(\varphi).
\end{eqnarray}

Summarizing, we have $$H_{u_0}(\varphi) \geq 0,$$
if
$$G_1'(u_0)\cdot \varphi=\theta,$$
and
$$G_2'(u_0) \cdot \varphi=v-\lambda G_2(u_0),$$
for some $v \leq \theta$ and $\lambda \geq 0.$

Define \begin{eqnarray}A&=&\{H_{u_0}(\varphi)+r, G_1'(u_0)\cdot \varphi, G_2'(u_0)\varphi-v+\lambda G_2(u_0)), \nonumber \\ &&
\varphi \in U, \;r \geq 0, v \leq \theta, \lambda \geq 0\}. \end{eqnarray}
From the convexity of $H_{u_0}$ (and the hypotheses on $G_1'(u_0)$ and $G_2'(u_0)$) we have that $A$ is a convex set (with a non-empty interior).

If $$G_1'(u_0)\cdot  \varphi=\theta,$$
and
$$G_2'(u_0)\cdot \varphi -v+\lambda G_2(u_0)=\theta,$$
with $v \leq \theta$ and $\lambda \geq 0$
then $$H_{u_0}(\varphi)\geq 0,$$ so that
$$H_{u_0}(\varphi)+r\geq 0,\; \forall r \geq 0.$$

From this and $$H_{u_0}(\theta)=0,$$ we have that $(0,\theta,\theta)$ is on the boundary of $A$. Therefore, by the Hahn-Banach theorem,
geometric form, there exists $$(\beta,z_1^*,z_2^*) \in \mathbb{R} \times Z_1^* \times Z_2^*$$
such that $$(\beta,z_1^*,z_2^*) \neq (0,\theta,\theta)$$ and
\begin{eqnarray}\label{eq35}\beta(H_{u_0}(\varphi)+r)&+&\langle G_1'(u_0)\cdot\varphi, z_1^* \rangle_{Z_1}
\nonumber \\ &+&\langle G_2'(u_0)\cdot \varphi-v+\lambda G_2(u_0),z_2^* \rangle_{Z_2} \geq 0, \end{eqnarray}
$\forall \;\varphi \in U, \;r \geq 0, \;v \leq \theta, \; \lambda \geq 0.$
Suppose $\beta=0$. Fixing all variable except $v$ we get $z_2^* \geq \theta$. Thus, for $\varphi=c \varphi_0$ with arbitrary $c \in \mathbb{R}$, $v=\theta, \lambda=0,$ if $z_2^* \neq \theta,$ then $\langle G_2'(u_0) \cdot \varphi_0,z_2^*\rangle_{Z_2}<0$ so that, letting $c \rightarrow +\infty,$ we get a contradiction through (\ref{eq35}), so that $z_2^*=\theta.$ Since $G_1'(u_0)$ is onto, a similar
reasoning lead us to $z_1^*=\theta$, which contradicts $(\beta,z_1^*,z_2^*) \neq (0,\theta,\theta).$

Hence, $\beta \neq 0,$ and fixing all variables except  $r$ we obtain $\beta>0$. There is no loss of generality in assuming $\beta=1$.

Again fixing all variables except  $v$, we obtain $z_2^* \geq \theta.$ Fixing all variables except  $\lambda$, since $G_2(u_0) \leq \theta$
we obtain $$\langle G_2(u_0), z_2^* \rangle_{Z_2}=0.$$

Finally, for $r=0,\; v=\theta, \;\lambda=0,$ we get
$$ H_{u_0}(\varphi)+\langle G_1'(u_0) \varphi, z_1^* \rangle_{Z_1}
+\langle G_2'(u_0)\cdot \varphi,z_2^* \rangle_{Z_2}\geq 0=\langle \varphi,\theta\rangle_U,\; \forall \varphi \in U.$$

From this,
$$ \theta \in \partial^0(F(u_0)+\langle G_1(u_0), z_1^* \rangle_{Z_1}
+\langle G_2(u_0),z_2^* \rangle_{Z_2})=\partial^0F(u_0)+[G_1'(u_0)^*](z_1^*)+[G_2'(u_0)^*](z_2^*),$$
so that there exists $u^* \in \partial^0F(u_0),$ such that $$u^*+[G_1'(u_0)^*](z_1^*)+[G_2'(u_0)^*](z_2^*)=\theta,$$
so that $$\langle \varphi,u^*\rangle_U+\langle \varphi, G_1'(u_0)^*(z_1^*)\rangle_U+\langle \varphi, G_2'(u_0)^*(z_2^*)\rangle_U=0,$$
that is,
$$\langle \varphi,u^*\rangle_U+\langle G_1'(u_0)\varphi, z_1^*\rangle_{Z_1}+\langle G_2'(u_0)\varphi, z_2^*\rangle_{Z_2}=0,\; \forall \varphi \in U.$$
The proof is complete.
\end{proof}

\section{Conclusion} In this article we have presented a survey on Lagrange multipliers theorems for non-smooth variational optimization in a general Banach space context. The results are based on standard tools of functional analysis, calculus of variations and optimization.

We emphasize, in the present article, no hypotheses concerning convexity are assumed and the results indeed are valid for such a more general Banach space context.

\end{document}